\documentclass[12pt,reqno]{amsart}

\usepackage{amssymb}
\usepackage{amsthm}
\usepackage{amsmath}
\usepackage{a4wide}
\usepackage{latexsym}
\usepackage{mathrsfs}
\usepackage[T1]{fontenc}
\usepackage[latin1]{inputenc}
\usepackage{enumitem}
\usepackage{color}

\newtheorem{theorem}{Theorem}[section]
\newtheorem{lemma}[theorem]{Lemma}
\newtheorem{proposition}[theorem]{Proposition}

\theoremstyle{definition}
\newtheorem{definition}[theorem]{Definition}

\newtheorem{example}[theorem]{Example}
\newtheorem{remark}[theorem]{Remark}

\numberwithin{equation}{section}

\newcommand{\N}{\mathbb{N}}

\newcommand{\R}{\mathbb{R}}

\renewcommand{\le}{\ensuremath{\leqslant}}

\renewcommand{\ge}{\ensuremath{\geqslant}}

\newcommand{\abs}[1]{\lvert#1\rvert}            
\newcommand{\norm}[1]{\lVert#1\rVert}           
\newcommand{\bigabs}[1]{\bigl\lvert#1\bigr\rvert} 
\newcommand{\bignorm}[1]{\bigl\lVert#1\bigr\rVert}

\newcommand{\term}[1]{{\textit{\textbf{#1}}}}   
\newcommand{\midcolon}{\::\:} 
\newcommand{\supp}{\operatorname{supp}}


\begin{document}
\title[Vector lattices admitting a p.h.c.\ function calculus]{Vector lattices admitting a positively homogeneous continuous function calculus}

\author[N.~J.~Laustsen]{Niels Jakob Laustsen}
\address{Department of Mathematics and Statistics, Fylde College, Lancaster University, Lancaster,
LA1 4YF, United Kingdom}
\email{n.laustsen@lancaster.ac.uk}

\author[V.~G.~Troitsky]{Vladimir G.~Troitsky}
\address{Department of Mathematical and Statistical Sciences, University of Alberta, Edmonton, Alberta T6G 2G1, Canada}
\email{troitsky@ualberta.ca}

\begin{abstract}
 We characterize the Archimedean vector lattices that admit a
 positively ho\-mo\-ge\-neous continuous function calculus by showing
 that the
 following two conditions are equivalent for each $n$-tuple $\boldsymbol{x} = (x_1,\ldots,x_n)\in X^n$,
 where~$X$ is an Archimedean vector lattice and $n\in\N$:
\begin{itemize}
\item there is a vector lattice homo\-mor\-phism
  \mbox{$\Phi_{\boldsymbol{x}}\colon H_n\to X$} such that
  \[  \Phi_{\boldsymbol{x}}(\pi_i^{(n)}) = x_i\qquad 
  (i\in\{1,\ldots,n\}), \] where~$H_n$ denotes the vector
  lattice of positively homogeneous, continuous, real-valued functions
  defined on~$\R^n$ and $\pi_i^{(n)}\colon\R^n\to\R$ is the
  $i^{\text{th}}$ coordinate projection;
\item there is a positive
  element $e\in X$ such that $e\geqslant\lvert x_1\rvert\vee\cdots\vee\lvert x_n\rvert$
  and the norm
\[  \lVert x\rVert_e = \inf\bigl\{
  \lambda\in[0,\infty)\:\colon\:\lvert x\rvert\le\lambda e\bigr\},
     \] 
defined for each $x$ in the order ideal~$I_e$ of~$X$ generated by~$e$,
is complete when restricted to the closed sub\-lattice of~$I_e$
generated by $x_1,\ldots,x_n$.
\end{itemize}
Moreover, we show that a vector space which admits a `sufficiently
strong' $H_n$-function calculus for each $n\in\N$ is automatically a
vector lattice, and we explore the situation in the
non-Archi\-me\-dean case by showing that some non-Archi\-me\-dean vector
lattices admit a positively ho\-mo\-ge\-neous continuous function
calculus, while others do not.
\end{abstract}
\maketitle
\section{Introduction and main results}
\noindent
Yudin~\cite{yudin} and Krivine~\cite{krivine} showed that every Banach
lattice admits a positively homogeneous continuous function calculus.
We refer to~\cite[Theorem~1.d.1]{LT2} for an easily accessible,
precise state\-ment of this result, which is a fundamental tool in the
study of Banach lattices, for in\-stance allowing the definition of
elements of the form $\bigl(\sum_{j=1}^n\abs{x_j}^p\bigr)^{1/p}$ for
$p\in(1,\infty)$ when\-ever $n\in\N$ and $x_1,\ldots,x_n$ belong to
some Banach lattice.  Buskes, de Pagter and van
Rooij \cite[Theorem~3.7]{BdPvR} have subsequently generalized the
theorem of Yudin and Krivine to the class of uniformly complete
Archimedean vector lattices. 

After most of the research that this note is based upon was carried
out, we learnt of a recent paper~\cite{BS} of Buskes and Schwanke in
which they study completions of Archimedean vector lattices with
respect to a given non-empty collection~$\mathcal{D}$ of
positively homogeneous continuous functions. Of particular interest in
the present context is their final result \cite[Corollary~3.18]{BS},
which states that for each such collection~$\mathcal{D}$, every
Archimedean vector lattice has a
$\mathcal{D}$\nobreakdash-com\-ple\-tion, and it is unique up to
vector lattice isomorphism.

The main aim of our work is to characterize the Archimedean vector
lattices that admit a positively homogeneous continuous function
calculus; we refer to Theorem~\ref{mainthm} for a precise statement of
this result. In combination with
Proposition~\ref{prop:twoexamples}\ref{sublatticeX}, it will in
particular show that this class is strictly larger than that of
uniformly complete Archimedean vector lattices.  In the context of the
work~\cite{BS} of Buskes and Schwanke discussed above,
Theorem~\ref{mainthm} can be viewed as providing an alternative
and perhaps more explicit description of the class of
$\mathcal{D}$\nobreakdash-com\-plete Archimedean vector lattices in
the special case where~$\mathcal{D}$ is the collection of \textsl{all}
positively homogeneous continuous functions.\smallskip

In order to state our results precisely, we must introduce some
notation and terminology.  All vector spaces and lattices are assumed
to be over the field~$\R$ of real numbers. A real-valued function~$f$
defined on a vector space~$X$ is \term{positively homogeneous} if
\mbox{$f(\lambda x) = \lambda f(x)$} for each $\lambda\in[0,\infty)$
  and each $x\in X$.  For $n\in\N$, we denote by~$H_n$ the vector
  lattice of positively homogeneous, continuous, real-valued functions
  defined on~$\R^n$. The $i^{\text{th}}$ coordinate projection \[
  \pi_i^{(n)}\colon\ (t_1,\ldots,t_n)\mapsto t_i,\quad \R^n\to\R, \]
  clearly belongs to~$H_n$ for each $i\in\{1,\ldots,n\}$.

We can now clarify what we  mean by `a positively homogeneous
continuous function calculus' for a general vector lattice.
\begin{definition}\label{defnphcfc}
  A vector lattice~$X$ admits a \term{positively homogeneous continuous
    function calculus} if, for each $n\in\N$ and each $n$-tuple $\boldsymbol{x} =
  (x_1,\ldots,x_n)\in X^n$, there is a vector lattice homo\-mor\-phism
  \mbox{$\Phi_{\boldsymbol{x}}\colon H_n\to X$} such that
  \begin{equation}\label{defnphcfcEq} \Phi_{\boldsymbol{x}}(\pi_i^{(n)}) = x_i\qquad 
  (i\in\{1,\ldots,n\}).
  \end{equation}
  In this case, we refer to the map
  $\boldsymbol{x}\mapsto\Phi_{\boldsymbol{x}}$ as a \term{positively
    homogeneous continuous function calculus} for~$X$.
\end{definition}

We shall next introduce a condition which turns out to be equivalent
to admitting a posi\-tive\-ly homogeneous continuous function
calculus. It involves the following standard notions.  For a positive
element~$e$ of a vector lattice~$X$, the set
\begin{equation}\label{OrderIdeal} I_e = \bigl\{ x\in X \midcolon\abs{x}\le\lambda
  e\ \text{for some}\ \lambda\in[0,\infty)\bigr\}
\end{equation} is the \term{order ideal generated by}~$e$, and
\begin{equation}\label{OrderIdealNorm}
  \norm{x}_e = \inf\bigl\{
  \lambda\in[0,\infty)\midcolon\abs{x}\le\lambda e\bigr\}\qquad (x\in
    I_e) \end{equation} defines a lattice seminorm on~$I_e$.  Suppose
that~$X$ is \term{Archimedean}, that is, whenever $x,y\in X^+$ satisfy
$nx\le y$ for each $n\in\N$, it follows that $x= 0$. Then
$\norm{\,\cdot\,}_e$ is a norm on~$I_e$.

\begin{definition}
Let~$X$ be an Archimedean vector lattice. Then:
\begin{itemize}
\item $X$ is \term{uniformly complete} if, for each positive element
  $e\in X$, the order ideal~$I_e$ is complete with respect to the norm
  $\norm{\,\cdot\,}_e$ given by~\eqref{OrderIdealNorm};
\item $X$ is \term{finitely uniformly complete} if, for each $n\in\N$
  and $x_1,\ldots,x_n\in X$, there is a positive element $e\in X$ such
  that $e\ge\abs{x_1}\vee\cdots\vee\abs{x_n}$ and the norm
  $\norm{\,\cdot\,}_e$ is complete on the closed sub\-lattice of
  $\bigl(I_e,\norm{\,\cdot\,}_e\bigr)$ generated by $x_1,\ldots,x_n$.
\end{itemize}
\end{definition}

The first of these two notions is standard, whereas the second appears
to be new. Clearly, the first implies the second;
Proposition~\ref{prop:twoexamples}\ref{sublatticeX} below will show
that they are not equivalent.

We are now ready to state our two main results. 

\begin{theorem}\label{mainthm}
Let~$X$ be an Archimedean vector lattice. Then~$X$ admits a positively
ho\-mo\-ge\-neous continuous function calculus if and only if~$X$ is
finitely uniformly complete.

When~$X$ is finitely uniformly complete, the positively homogeneous
continuous function calculus is unique (in the sense that for each
$n\in\N$ and $\boldsymbol{x}\in X^n$, there is only one vector lattice
homomorphism \mbox{$\Phi_{\boldsymbol{x}}\colon H_n\to X$} which
satisfies~\eqref{defnphcfcEq}), and
\begin{equation}\label{compositionsEq}
  \Phi_{(\Phi_{\boldsymbol{x}}(f_1),\ldots,\Phi_{\boldsymbol{x}}(f_m))}(g)
  = \Phi_{\boldsymbol{x}}\bigl(g\circ(f_1\times\cdots\times
  f_m)\bigr)
\end{equation} 
for each $m,n\in\N,$ $\boldsymbol{x}\in X^n$, $f_1,\ldots,f_m\in H_n$
and $g\in H_m,$ where $f_1\times\cdots\times f_m\colon
\R^n\to \R^m$ is the function
defined by
\begin{equation}\label{f1timesdotsfmDefn} 
(f_1\times\cdots\times f_m)(t) = \bigl(f_1(t),\ldots,f_m(t)\bigr)\qquad
(t\in\R^n). \end{equation} 
\end{theorem}

Writing $f(\boldsymbol{x})$ for $\Phi_{\boldsymbol{x}}(f)$,
\eqref{compositionsEq} takes the more suggestive form
\begin{equation}\label{compEqAlt} g\bigl(f_1(\boldsymbol{x}),\ldots, f_m(\boldsymbol{x})\bigr) =
\bigl(g\circ(f_1\times\cdots\times f_m)\bigr)(\boldsymbol{x}). \end{equation}

Our other main result states that a \textsl{vector space} which admits
a `sufficiently strong' function calculus is automatically a
\textsl{vector lattice} with a positively homogeneous continuous
function calculus.

\begin{theorem}\label{mainthm2}
  Let $X$ be a vector space and suppose that, for each $n\in\N$ and
  each $n$-tuple $\boldsymbol{x}\in X^n,$ there is a linear map
  $\Phi_{\boldsymbol{x}}\colon H_n\to X$ which satisfies
  conditions~\eqref{defnphcfcEq} and~\eqref{compositionsEq}. Then~$X$
  admits the structure of a vector lattice, and
  $\boldsymbol{x}\mapsto\Phi_{\boldsymbol{x}}$ is a positively
  homogeneous con\-tinuous function calculus for~$X$ with respect to
  this lattice structure.  Moreover, $X$ is Ar\-chi\-me\-dean if and only if
  $\ker\Phi_{\boldsymbol{x}}$ is closed in~$H_n$ for each $n\in\N$ and
  $\boldsymbol{x}\in X^n,$ where~$H_n$ is given the topology obtained
  by identifying it with~$C(S_{\ell_\infty^n})$
  (see~\eqref{HnequalsCSlinftyn} for details of this iden\-ti\-fi\-ca\-tion).
\end{theorem}
The notion of (finite) uniform completeness does not extend easily to
the non-Ar\-chi\-medean setting. Indeed, if $x,y\in X^+\setminus\{0\}$
satisfy $nx\leqslant y$ for each $n\in\N$, then for any $e\in X^+$
such that $e\geqslant x\vee y$, we have $\norm{x}_e = 0$ because
$nx\leqslant y\leqslant e$ for each $n\in\N$. Hence $\norm{\,\cdot\,}_e$
is only a semi\-norm (even when it is restricted to the sub\-lattice generated
by~$x$ and~$y$). 

This raises the question whether a non-Archimedean vector lattice may
admit a positively homogeneous continuous function calculus. We shall
address it in Section~\ref{section2}, where on the one hand 
Proposition~\ref{prop:twoexamples}\ref{orderidealJ} will show that
certain non-Archimedean vector lattices do ad\-mit a positively
homogeneous continuous function calculus which
satisfies~\eqref{compositionsEq}, while on the other
Example~\ref{ex:R2lexHasNoPhcfc} will exhibit a non-Archimedean vector
lattice that does not admit any positively homogeneous continuous
function calculus.  We observe that the former of these two results
implies that the condition in the final clause of
Theorem~\ref{mainthm2} is not always satisfied.

\section{Proofs of Theorems~\ref{mainthm} and~\ref{mainthm2}}
\noindent
For a topological space~$K$, we write $C(K)$ for the vector lattice of
continuous, real-valued functions defined on~$K$. This is of course a
Banach lattice with respect to the supremum
norm~$\norm{\,\cdot\,}_\infty$ when~$K$ is a compact Hausdorff space. We use the symbol~$\boldsymbol{1}$ to denote the constant
function~$1$ defined on~$K$.

\begin{lemma}\label{phcfcImpliesFLC}
Let $T\colon C(K)\to X$ be a vector lattice homomorphism, where~$K$ is
a compact Haus\-dorff space and~$X$ is an Archimedean vector lattice,
and set $e = T(\boldsymbol{1})\in X^+$. Then $T[C(K)]\subseteq I_e,$
$T$ is continuous with operator norm at most~$1$ when regarded as a
map into~$\bigl(I_e,\norm{\,\cdot\,}_e\bigr),$ and the restriction of
the norm~$\norm{\,\cdot\,}_e$ to~$T[(C(K)]$ is complete.
\end{lemma}

\begin{proof}
  Each $f\in C(K)$ satisfies $\abs{f}\le
  \norm{f}_\infty\boldsymbol{1}$, so as~$T$ is a vector lattice
  homomorphism, we have
  \begin{equation}\label{phcfcImpliesFLCeq2} 
    \norm{f}_\infty e = T\bigl(\norm{f}_\infty\boldsymbol{1}\bigr)
    \ge T\bigl(\abs{f}\bigr) =\bigabs{T(f)}.
  \end{equation}
  This shows that $T(f)\in I_e$ with
  $\bignorm{T(f)}_e\le\norm{f}_\infty$, so that $T[C(K)]\subseteq I_e$
  and~$T$ is continuous with operator norm at most~$1$ when regarded
  as a map into~$\bigl(I_e,\norm{\,\cdot\,}_e\bigr)$.

  We shall verify that the restriction of the
  norm~$\norm{\,\cdot\,}_e$ to~$T[(C(K)]$ is complete by showing that
  each absolutely convergent series $\sum_{n=1}^\infty x_n$ in
  $\bigl(T[(C(K)],\norm{\,\cdot\,}_e\bigr)$ converges.  It suffices to
  consider the case where~$x_n$ is positive for each $n\in\N$ because
  $\norm{\,\cdot\,}_e$ is a lattice norm. Then we can take $f_n\in
  C(K)^+$ such that $Tf_n=x_n$. Set $g_n=f_n\wedge\norm{x_n}_e\boldsymbol{1}$
  and observe that $\norm{g_n}_\infty\le\norm{x_n}_e$, so that the
  series $\sum_{n=1}^\infty g_n$ is absolutely convergent and
  therefore convergent in $C(K)$; denote its sum by~$g$. We see that
  $Tg_n=x_n\wedge\norm{x_n}_ee=x_n$ because
  $x_n\le\norm{x_n}_ee$. Since~$T$ is continuous and linear, we conclude
  that the series $\sum_{n=1}^\infty x_n$ converges to~$Tg$ in
  $\bigl(T[C(K)],\norm{\,\cdot\,}_e\bigr)$, and the result follows.
\end{proof}  

For $n\in\N$,  the unit sphere 
\[ S_{\ell_\infty^n} = \bigl\{ (t_1,\ldots,t_n)\in\R^n\midcolon \max_{1\le j\le
  n}\abs{t_j} = 1\bigr\} \] of the Banach space~$\ell_\infty^n$ is a compact
metric space with respect to the metric~$d$ induced by the norm, that is,
\[ d\bigl((s_1,\ldots,s_n),(t_1,\ldots,t_n)\bigr) = \max_{1\le j\le
  n}\abs{s_j-t_j}. \]
It is well known and easy to see that
the restriction map
\begin{equation}\label{HnequalsCSlinftyn}
  f\mapsto f\!\!\upharpoonright_{S_{\ell_\infty^n}}
\end{equation}
is a vector lattice isomorphism of~$H_n$ onto~$ C(S_{\ell_\infty^n})$,
where we recall that~$H_n$ denotes the sub\-lattice of~$C(\R^n)$ of
positively homogeneous functions. Hence we may identify~$H_n$ with the
Banach lattice~$C(S_{\ell_\infty^n})$. Although we do not require this
result, we remark that de Pagter and
Wickstead~\cite[Proposition~5.3]{dePW} have shown that this Banach
lattice is isomorphic to the free Banach lattice on~$n$ generators.

The following result can be viewed as a generalization of 
\cite[Theorem~3.7]{BdPvR}. 

\begin{proposition}\label{PropStep1ofMainThm}
Let~$X$ be an Archimedean vector lattice, and let 
$\boldsymbol{x}=(x_1,\ldots,x_n)\in X^n$ for some $n\in\N$. 
Then the following three conditions are equivalent: 
\begin{enumerate}[label={\normalfont{(\alph*)}}]
\item\label{PropStep1ofMainThm1} there is a vector lattice
  homo\-mor\-phism \mbox{$\Phi_{\boldsymbol{x}}\colon H_n\to X$} which
  satisfies~\eqref{defnphcfcEq}{\normalfont{;}}
\item\label{PropStep1ofMainThm2} the norm $\norm{\,\cdot\,}_e$ is
  complete on the closed sub\-lattice
  of~$\bigl(I_e,\norm{\,\cdot\,}_e\bigr)$ generated by
  the elements $x_1,\ldots,x_n,$ where $e = \abs{x_1}\vee\cdots\vee\abs{x_n};$
\item\label{PropStep1ofMainThm3} there is a positive element $e\in X$
  such that $e\ge\abs{x_1}\vee\cdots\vee\abs{x_n}$ and the norm
  $\norm{\,\cdot\,}_e$ is complete on the closed sub\-lattice
  of~$\bigl(I_e,\norm{\,\cdot\,}_e\bigr)$ generated by
  $x_1,\ldots,x_n$.
\end{enumerate}
When one and hence all three of these conditions are satisfied, the
vector lattice homo\-mor\-phism \mbox{$\Phi_{\boldsymbol{x}}\colon
  H_n\to X$} satisfying~\eqref{defnphcfcEq} is unique.
\end{proposition}

\begin{proof} Throughout the proof we shall freely
identify~$H_n$ with~$C(S_{\ell_\infty^n})$ via the
map~\eqref{HnequalsCSlinftyn}.

\ref{PropStep1ofMainThm1}$\Rightarrow$\ref{PropStep1ofMainThm2}. 
Suppose that $\Phi_{\boldsymbol{x}}\colon C(S_{\ell_\infty^n})\to X$ is a vector
lattice homo\-mor\-phism which sa\-tis\-fies~\eqref{defnphcfcEq}.
Then,
applying~$\Phi_{\boldsymbol{x}}$ to the identity
\begin{equation}\label{EqMaxofpiis}
\bigabs{\pi_1^{(n)}\!\!\upharpoonright_{S_{\ell_\infty^n}}\!}\vee\cdots\vee
\bigabs{\pi_n^{(n)}\!\!\upharpoonright_{S_{\ell_\infty^n}}\!} =
\boldsymbol{1}, \end{equation} we obtain
\[ \Phi_{\boldsymbol{x}}(\boldsymbol{1}) = 
\bigabs{\Phi_{\boldsymbol{x}}(\pi_1^{(n)}\!\!\upharpoonright_{S_{\ell_\infty^n}}\!)}\vee\cdots\vee
\bigabs{\Phi_{\boldsymbol{x}}(\pi_n^{(n)}\!\!\upharpoonright_{S_{\ell_\infty^n}}\!)}
= \abs{x_1}\vee\cdots\vee \abs{x_n} = e. \] Hence Lemma~\ref{phcfcImpliesFLC}
implies that $\Phi_{\boldsymbol{x}}[C(S_{\ell_\infty^n})]\subseteq
I_e$ and the restriction of the norm~$\norm{\,\cdot\,}_e$
to~$\Phi_{\boldsymbol{x}}[C(S_{\ell_\infty^n})]$ is complete, and
therefore~\ref{PropStep1ofMainThm2} is satisfied because
$x_1,\ldots,x_n\in \Phi_{\boldsymbol{x}}[C(S_{\ell_\infty^n})]$.

\ref{PropStep1ofMainThm2}$\Rightarrow$\ref{PropStep1ofMainThm3}. This is clear.

\ref{PropStep1ofMainThm3}$\Rightarrow$\ref{PropStep1ofMainThm1}.
Suppose that $e\in X^+$ satisfies $e\ge
\abs{x_1}\vee\cdots\vee\abs{x_n}$ and the
norm~$\norm{\,\cdot\,}_e$ is complete on the closed
sub\-lattice~$X_{\boldsymbol{x}}$ of~$I_e$ generated
by~$x_1,\ldots,x_n$. Then 
$\bigl(X_{\boldsymbol{x}},\norm{\,\cdot\,}_e\bigr)$ is a Banach
lattice containing~$x_1,\ldots,x_n$, so the Yudin/Krivine Theorem as
it is stated in \cite[Theorem~1.d.1]{LT2} implies that there is a
vector lattice homomorphism $\Phi_{\boldsymbol{x}}\colon H_n\to
X_{\boldsymbol{x}}$ which satisfies~\eqref{defnphcfcEq}. This remains
true if we consider~$\Phi_{\boldsymbol{x}}$ as a map into the larger
codomain~$X$, so that~\ref{PropStep1ofMainThm1} is satisfied.

To prove the final clause, suppose that
\ref{PropStep1ofMainThm1}--\ref{PropStep1ofMainThm3} hold, and let
$\Phi_{\boldsymbol{x}}\colon C(S_{\ell_\infty^n})\to X$ be a vector
lattice homo\-mor\-phism which satisfies~\eqref{defnphcfcEq}.
Then~$\Phi_{\boldsymbol{x}}$ is uniquely determined on the
sub\-lattice of~$C(S_{\ell_\infty^n})$ generated by
\mbox{$\pi_1^{(n)}\!\!\upharpoonright_{S_{\ell_\infty^n}},\dots,
  \pi_n^{(n)}\!\!\upharpoonright_{S_{\ell_\infty^n}}$}.  The
Stone--Weier\-strass Theorem implies that this sub\-lattice is dense
in $C(S_{\ell_\infty^n})$ because it separates the points
of~$S_{\ell_\infty^n}$ and con\-tains~$\boldsymbol{1}$ by~\eqref{EqMaxofpiis}.
As in the proof of
\ref{PropStep1ofMainThm1}$\Rightarrow$\ref{PropStep1ofMainThm2} above,
we see that Lemma~\ref{phcfcImpliesFLC} applies; it shows
that~$\Phi_{\boldsymbol{x}}$ is continuous when regarded as an
operator into~$\bigl(I_e,\norm{\,\cdot\,}_e\bigr)$,
 and therefore~$\Phi_{\boldsymbol{x}}$ is uniquely
determined on all of~$C(S_{\ell_\infty^n})$.
\end{proof}

\begin{proof}[Proof of Theorem~{\normalfont{\ref{mainthm}}}]
The equivalence of conditions~\ref{PropStep1ofMainThm1}
and~\ref{PropStep1ofMainThm3} in Proposition~\ref{PropStep1ofMainThm}
implies immediately that~$X$ admits a positively homogeneous
continuous function calculus if and only if it is finitely uniformly
complete.

To prove~\eqref{compositionsEq}, we begin by remarking that for
$m,n\in\N$ and $f_1,\ldots,f_m\in H_n$, the function
$f_1\times\cdots\times f_m$ given by~\eqref{f1timesdotsfmDefn} is
continuous and the composition $g\circ(f_1\times\cdots\times f_m)$ is
positively homogeneous for each $g\in H_m$. Hence, for
$\boldsymbol{x}\in X^n$, we have a map
\[ g\mapsto \Phi_{\boldsymbol{x}}\bigl(g\circ(f_1\times\cdots\times
f_m)\bigr),\quad H_m\to X. \] This map is a vector lattice
homomorphism which maps $\pi_i^{(m)}$ to $\Phi_{\boldsymbol{x}}(f_i)$
for each $i\in\{1,\ldots,m\}$ because
$\pi_i^{(m)}\circ(f_1\times\cdots\times f_m) = f_i$, and it is
therefore equal
to~$\Phi_{(\Phi_{\boldsymbol{x}}(f_1),\ldots,\Phi_{\boldsymbol{x}}(f_m))}$
by the uniqueness statement in the last clause of
Proposition~\ref{PropStep1ofMainThm}.
\end{proof}

We shall next prove Theorem~\ref{mainthm2}. This will involve the following easy and undoubtedly well-known lemma.

\begin{lemma}\label{sublatticeclosure}
Let $Y$ be a sub\-lattice of a normed vector lattice~$X$, and suppose
that \[ \overline{Y}\cap X^+\subseteq Y. \] Then~$Y$ is closed.   
\end{lemma}

\begin{proof}
Let $y\in\overline{Y}$. Since~$\overline{Y}$ is a sub\-lattice,
$y^\pm\in\overline{Y}$, and they are both positive by their
definitions. Hence $y^\pm\in Y$ by the assumption, and therefore $y =
y^+-y^-\in Y$.
\end{proof}

\begin{proof}[Proof of Theorem~{\normalfont{\ref{mainthm2}}}] 
  Throughout this proof, we shall write $f(\boldsymbol{x})$ instead of
  $\Phi_{\boldsymbol{x}}(f)$. In this notation, the linearity of
  $\Phi_{\boldsymbol{x}}$ translates into the statement
  \begin{equation}\label{linearityEq}
  (f+\lambda g)(\boldsymbol{x}) = f(\boldsymbol{x}) +\lambda\,
   g(\boldsymbol{x})\qquad (n\in\N,\, \boldsymbol{x}\in X^n,\,
    f,g\in H_n,\, \lambda\in\R),
  \end{equation}
  while~\eqref{defnphcfcEq} becomes
  \begin{equation}\label{eq1.1alt} \pi_i^{(n)}(x_1,\ldots,x_n) = x_i\qquad
  (n\in\N,\,i\in\{1,\ldots,n\},\,x_1,\ldots,x_n\in X), \end{equation} and~\eqref{compEqAlt} replaces~\eqref{compositionsEq}.

  The map $\sigma\colon (t_1,t_2)\mapsto t_1\vee t_2,\, \R^2\to \R$,
  belongs to~$H_2$, so that $\sigma(x_1,x_2)$ defines an element
  of~$X$ for each pair $x_1,x_2\in X$. Consequently, we may define a
  relation~$\le$ on~$X$ by
  \[ x_1\le x_2\quad \Longleftrightarrow\quad \sigma(x_1,x_2)=x_2\qquad\quad (x_1,x_2\in X). \] 
  Our first aim is to show that this relation is a partial order on~$X$.

  \textsc{Reflexivity.}  The fact that $\sigma(t,t)=t$ for every
  $t\in\R$ implies that
  $\sigma\circ\bigl(\pi^{(2)}_1\times\pi^{(2)}_1\bigr)=\pi^{(2)}_1$ in
  $H_2$. Hence, by~\eqref{compEqAlt} and~\eqref{eq1.1alt}, we obtain
  \begin{equation}\label{reflEq}
  \sigma(x,x)=x\qquad (x\in X), 
  \end{equation}
  which shows that $x\le x$ for each $x\in X$, as required.

  \textsc{Anti-symmetry.}  Since $\sigma(t_1,t_2)=\sigma(t_2,t_1)$ for
  every pair $t_1,t_2\in\R$, we have
  $\sigma=\sigma\circ(\pi^{(2)}_2\times\pi^{(2)}_1)$
  in~$H_2$. Combining this identity with~\eqref{compEqAlt}
  and~\eqref{eq1.1alt}, we see that
  \begin{equation}\label{symmofsigmaEq}
  \sigma(x_1,x_2)=\sigma\bigl(\pi^{(2)}_2(x_1,x_2),\pi^{(2)}_1(x_1,x_2)\bigr)
  =\sigma(x_2,x_1)\qquad (x_1,x_2\in X).
  \end{equation}
  Now suppose that $x_1,x_2\in X$ satisfy $x_1\le x_2$ and $x_2\le
  x_1$. Then $\sigma(x_1,x_2)=x_2$ and $\sigma(x_2,x_1)=x_1$, so that
  $x_1=x_2$ by~\eqref{symmofsigmaEq}, as required.

  \textsc{Transitivity.}  The associativity of~$\vee$ means that
  $\sigma\bigl(t_1,\sigma(t_2,t_3)\bigr)
  =\sigma\bigl(\sigma(t_1,t_2),t_3\bigr)$ for every
  $t_1,t_2,t_3\in\R$. Hence the identity 
  \[ \sigma\circ\Bigl(\pi^{(3)}_1\times
   \bigl(\sigma\circ(\pi^{(3)}_2\times\pi^{(3)}_3)\bigr)\Bigr) =
   \sigma\circ\Bigl(\bigl(\sigma\circ(\pi^{(3)}_1\times\pi^{(3)}_2)\bigr)
   \times\pi^{(3)}_3\Bigr) \] holds in~$H_3$, from  which we deduce that
  \begin{equation}\label{assoc}
    \sigma\bigl(x_1,\sigma(x_2,x_3)\bigr) =
    \sigma\bigl(\sigma(x_1,x_2),x_3\bigr)\qquad (x_1,x_2,x_3\in X)
  \end{equation}
  by~\eqref{compEqAlt} and~\eqref{eq1.1alt}.
  Now suppose that
  $x_1,x_2,x_3\in X$ satisfy $x_1\le x_2$ and $x_2\le x_3$. Then we
  have $\sigma(x_1,x_2)=x_2$ and $\sigma(x_2,x_3)=x_3$. Substituting
  these identities into~\eqref{assoc}, we obtain
  $\sigma(x_1,x_3)
    =\sigma(x_2,x_3)=x_3$, 
    which shows that $x_1\le x_3$, as required.\smallskip

  Having thus established that~$\le$ is a partial order, we shall next
  show that each pair $(x_1,x_2)$ of elements of~$X$ has a supremum
  with respect to~$\le$, and it is given by~$\sigma(x_1,x_2)$. To this end, we observe that~\eqref{assoc} and~\eqref{reflEq} imply that
  \[ \sigma\bigl(x_1,\sigma(x_1,x_2)\bigr) = \sigma\bigl(\sigma(x_1,x_1),x_2\big) = \sigma(x_1,x_2), \] 
  so that $x_1\le \sigma(x_1,x_2)$. A similar argument, using
  also~\eqref{symmofsigmaEq}, shows that $x_2\le\sigma(x_1,x_2)$, and
  therefore $\sigma(x_1,x_2)$ is an upper bound of the pair
  $(x_1,x_2)$.  To show that it is the \textsl{least} upper bound,
  suppose that $y\in X$ satisfies $x_1\le y$ and $x_2\le y$, so that
  $\sigma(x_1,y)=y=\sigma(x_2,y)$. Then we have $\sigma(x_1,x_2)\le y$
  because~\eqref{assoc} implies that
  \[ \sigma\bigl(\sigma(x_1,x_2),y\bigr) = \sigma\bigl(x_1,\sigma(x_2,y)\bigr)
    =\sigma(x_1,y)=y, \] as required.

  It remains to verify that $\le$ is positively homogeneous and
  translation-invariant.

  \textsc{Positive homogeneity.}  Suppose that $x_1,x_2\in X$ satisfy
  $x_1\le x_2$, and let $\lambda\in[0,\infty)$. The positive
    homogeneity of~$\sigma$ translates into the identity
$\lambda\sigma = \sigma\circ(\lambda\pi_1^{(2)}\times\lambda\pi_2^{(2)})$ in~$H_2$, and therefore we have 
\[ \sigma(\lambda x_1,\lambda x_2) = (\lambda\sigma)(x_1,x_2) = \lambda\,\sigma(x_1,x_2) = \lambda x_2 \] 
by~\eqref{compEqAlt}, \eqref{linearityEq} and~\eqref{eq1.1alt}. This
shows that $\lambda x_1\le\lambda x_2$.

\textsc{Translation invariance.} Since $\sigma(t_1+t_3,t_2+t_3) =
\sigma(t_1,t_2)+t_3$ for every $t_1,t_2,t_3\in\R$, we have 
\begin{equation}\label{translatinvEq1} 
\sigma\circ\bigl((\pi_1^{(3)}+\pi_3^{(3)})\times(\pi_2^{(3)}+\pi_3^{(3)})\bigr)
= \sigma\circ(\pi_1^{(3)}\times\pi_2^{(3)}) +
\pi_3^{(3)} \end{equation} in~$H_3$. Let $x_1,x_2,x_3\in X$, suppose
that $x_1\le x_2$, and set $\boldsymbol{x} = (x_1,x_2,x_3)$. Then we
have
\begin{alignat*}{2}
\sigma(x_1+x_3,x_2+x_3) &=
\bigl(\sigma\circ(\pi_1^{(3)}\times\pi_2^{(3)}) +
\pi_3^{(3)}\bigr)(\boldsymbol{x})\quad &&\text{by \eqref{compEqAlt},
  \eqref{linearityEq}, \eqref{eq1.1alt} and
  \eqref{translatinvEq1}}\\ &= \sigma(x_1,x_2) + x_3\quad &&\text{by
  \eqref{compEqAlt}, \eqref{linearityEq} and
  \eqref{eq1.1alt}}\\ &=x_2+x_3\quad &&\text{because}\ x_1\le x_2,
\end{alignat*}
so that $x_1+x_3\le x_2+x_3$, as required. 

This completes the proof that~$X$ is a vector lattice with respect to
the order~$\le$.\smallskip

To show that~$\Phi_{\boldsymbol{x}}$ is a vector lattice homomorphism
for every $n\in\N$ and $\boldsymbol{x}\in X^n$, take $f,g\in
H_n$. Then $f\vee g=\sigma\circ(f\times g)$ in $H_n$, so that $(f\vee
g)(\boldsymbol{x}) = \sigma(f(\boldsymbol{x}), g(\boldsymbol{x}))$
by~\eqref{compEqAlt}. Recalling the convention that
$\Phi_{\boldsymbol{x}}(h) = h(\boldsymbol{x})$ for $h\in H_n$ and that
the supremum in~$X$ is given by $x_1\vee x_2 =\sigma(x_1,x_2)$ for
$x_1,x_2\in X$, we see that this means that~$\Phi_{\boldsymbol{x}}$
preserves suprema.  Hence the conclusion follows
because~$\Phi_{\boldsymbol{x}}$ is linear by assumption.
\smallskip

We shall prove the forward implication ($\Rightarrow$) of the final
statement by contraposition. Suppose that $\ker\Phi_{\boldsymbol{x}}$
is not closed in~$H_n$ for some $n\in\N$ and $\boldsymbol{x}\in
X^n$. By Lemma~\ref{sublatticeclosure}, we can choose $f\in
H_n^+\cap\overline{\ker\Phi_{\boldsymbol{x}}}\setminus
\ker\Phi_{\boldsymbol{x}}$, which implies that
$\Phi_{\boldsymbol{x}}(f)\in X^+\setminus\{0\}$. For each $m\in\N$,
take $f_m\in \ker\Phi_{\boldsymbol{x}}$ such that $\norm{f -
f_m}_\infty\le 1/m$. Then we have $m(f-f_m)\le\boldsymbol{1}$, so
that
\[ m\Phi_{\boldsymbol{x}}(f) = \Phi_{\boldsymbol{x}}\bigl(m(f-f_m)\bigr)\le \Phi_{\boldsymbol{x}}(\boldsymbol{1})\qquad (m\in\N). \]
This shows that~$X$ is not Archimedean.

Conversely, suppose that $\ker\Phi_{\boldsymbol{x}}$ is closed
in~$H_n$ for each $n\in\N$ and $\boldsymbol{x}\in X^n$. Further,
suppose that $x_1,x_2\in X^+$ satisfy $mx_1\le x_2$ for each $m\in\N$,
and set \mbox{$\boldsymbol{x} = (x_1,x_2)\in X^2$}. The Fundamental
Isomorphism Theorem implies that the vector
lattices~$\Phi_{\boldsymbol{x}}[H_2]$ and
$H_2/\ker\Phi_{\boldsymbol{x}}$ are isomorphic. Since the order
ideal~$\ker \Phi_{\boldsymbol{x}}$ is closed by the assumption, the
quotient $H_2/\ker\Phi_{\boldsymbol{x}}$ is a Banach lattice and
thus Archimedean. Therefore~$\Phi_{\boldsymbol{x}}[H_2]$ is also
Ar\-chi\-me\-dean, from which we conclude that $x_1 = 0$ because
$x_1,x_2\in\Phi_{\boldsymbol{x}}[H_2]$.
\end{proof}

\section{Examples}\label{section2} 

\noindent The purpose of this section is to present three examples
that complement Theorems~\ref{mainthm} and~\ref{mainthm2}. The first
of these examples shows that there are finitely uniformly complete
Archimedean vector lattices which are not uniformly complete, while
the second and third explore the situation for non-Archi\-medean
vector lattices by demonstrating that some non-Archimedean vector
lattices admit a positively homogeneous continuous function calculus
satisfying~\eqref{compositionsEq}, whereas others do not admit any
positively homogeneous continuous function calculus.

Since the first two of these examples work in the generality of an
arbitrary infinite compact Hausdorff space, we have chosen to formally
label them `proposition' instead of `example', even though they serve
as illustrative examples in the context of this paper. Moreover, as
they share a common framework and have several parts of their proofs
in common, we shall state them as a single proposition rather than
two.

\begin{proposition}\label{prop:twoexamples}
Let~$K$ be a compact Hausdorff space of
infinite cardinality, and choose 
an accumulation point $s_0\in K$ of a countably infinite subset of~$K$. 
\begin{enumerate}[label={\normalfont{(\roman*)}}]
\item\label{sublatticeX}
The set
\begin{equation}\label{sublatticeX:Eq1}  X =
\bigl\{ f\in C(K)\midcolon f\!\!\upharpoonright_U\: \text{is constant
  for some neighbourhood}\ U\ \text{of}\ s_0\bigr\} \end{equation} is
a proper, dense sub\-lattice of~$C(K)$ and thus Archimedean, and
$X$ is finitely uniformly complete, but not uniformly complete.
\item\label{orderidealJ}
The set 
\[ J = \bigl\{ f\in C(K)\midcolon f\!\!\upharpoonright_U\, = 0\ \text{for
  some neighbourhood}\ U\ \text{of}\ s_0\bigr\} \] is an order ideal
of~$C(K);$ it is not closed, and the quotient $C(K)/J$ is a
non-Archi\-me\-dean vector lattice which admits a positively
homogeneous function calculus that satisfies~\eqref{compositionsEq}.
\end{enumerate}
\end{proposition}

\begin{proof}
Being infinite, $K$ contains a countably infinite subset~$S$, and the
compactness of~$K$ implies that~$S$ has an accumulation point~$s_0\in
K$.

\ref{sublatticeX}. It is straightforward to verify that~$X$ is a
sub\-lattice of~$C(K)$.

To show that $X\neq C(K)$, we begin by enumerating $S\setminus\{s_0\}$
as $\{s_m\midcolon m\in\N\}$, where $s_m\ne s_n$ for $m\ne n$. Then,
for each $m\in\N$, Urysohn's Lemma implies that there is a
continuous function $h_m\colon K\to [0,1]$ with $h_m(s_0) = 0$ and
$h_m(s_m) = 1$. Define 
\begin{equation}\label{defn:functionh}
h= \sum_{m=1}^\infty 2^{-m}h_m\in C(K)^+.
\end{equation}  For
each neighbourhood~$U$ of~$s_0$, we can choose $m\in\N$ such that
$s_m\in U$, and we have
\begin{equation}\label{eq:hsn} h(s_m)\geqslant \frac{1}{2^m}>0=h(s_0).
\end{equation}
This shows that $h\!\!\upharpoonright_U$ is not constant, so that
$h\notin X$.

Next, to establish the density of~$X$ in~$C(K)$, we shall prove that
for each $\varepsilon>0$ and $f\in C(K)$, we can find a
neighbourhood~$U$ of~$s_0$ and a function $g\in C(K)$ such that
\begin{equation}\label{eqn:fg}
  \norm{f-g}_\infty\leqslant\varepsilon\qquad\text{and}\qquad g(u) =
  f(s_0)\qquad (u\in U).
\end{equation}
Indeed, since~$f$ is continuous at~$s_0$, we can choose an open
neighbourhood~$V$ of~$s_0$ such that $\abs{f(v) -
f(s_0)}\leqslant\varepsilon$ for each $v\in V$. Being compact, $K$ is
regular, so there is an open neighbourhoood~$U$ of~$s_0$ with
$\overline{U}\subseteq V$. Urysohn's Lemma then produces a continuous
function $k\colon K\to[0,1]$ such that $k(u) = 1$ for each $u\in
\overline{U}$ and $k(t) = 0$ for each $t\in K\setminus V$. Set $g =
f(s_0)k + (1-k)f\in C(K)$. Clearly $g(u) = f(s_0)$ for each $u\in
\overline{U}$, and
\[ \bigabs{f(v)-g(v)} = 
\bigabs{f(v)-f(s_0)}\,\bigabs{k(v)}\leqslant \varepsilon\qquad
(v\in V)\qquad\text{and}\qquad g\!\!\upharpoonright_{K\setminus V} =
f\!\!\upharpoonright_{K\setminus V}, \] which implies
that~\eqref{eqn:fg} is satisfied.

Since~$X$ contains the constant functions, the order
ideal~$I_{\boldsymbol{1}}$ of~$X$ generated by the constant
function~$\boldsymbol{1}$ is equal to~$X$, and the associated
norm~$\norm{\,\cdot\,}_{\boldsymbol{1}}$ given by~\eqref{OrderIdealNorm}
is equal to the su\-pre\-mum norm~$\norm{\,\cdot\,}_\infty$. This implies
that $\bigl(I_{\boldsymbol{1}},\norm{\,\cdot\,}_{\boldsymbol{1}}\bigr)$ is not complete
because~$X$ is a proper, dense subset of~$C(K)$, and therefore~$X$ is
not uniformly complete.

To verify that~$X$ is finitely uniformly complete, let
$f_1,\ldots,f_n\in X$ for some $n\in\N$, and take a neighbourhood~$U$
of~$s_0$ such
that the restrictions $f_1\!\!\upharpoonright_U,\ldots,f_n\!\!\upharpoonright_U$ are
all constant. We may suppose that $f_1,\ldots,f_n$ are not all~$0$, so
that $c := \max_{1\le j\le n}\norm{f_j}_\infty>0$. Then $e :=
c\boldsymbol{1}$ satisfies $e\ge\abs{f_1}\vee\cdots\vee\abs{f_n}$ and
$\norm{g}_e = c^{-1}\norm{g}_\infty$ for each $g\in X$. Moreover, the set
  \[ Y =  \bigl\{ f\in C(K)\midcolon
  f\!\!\upharpoonright_U\ \text{is constant}\bigr\} \] is a closed
  sub\-lattice of~$C(K)$ such that $e,f_1,\ldots,f_n\in
  Y\subseteq X$. In particular, the norm~$\norm{\,\cdot\,}_e$ is
  complete on~$Y$ and therefore also on the closed sub\-lattice
  generated by $f_1,\ldots,f_n$. This proves that~$X$ is finitely
  uniformly complete.

\ref{orderidealJ}. It is easily checked that~$J$ is an order ideal
of~$C(K)$. We claim that its
closure is given by
\begin{equation}\label{eqn:closureofJ} \overline{J} = 
\bigl\{ f\in C(K)\midcolon f(s_0)=0\bigr\}. \end{equation} Indeed, the
set on the right-hand side is clearly a closed set
containing~$J$, and~\eqref{eqn:fg} implies that each~$f\in C(K)$ with
$f(s_0) = 0$ can be approximated arbitrarily well by elements of~$J$.
This proves~\eqref{eqn:closureofJ}, from which we deduce that~$J$ is
not closed because~\eqref{eq:hsn} shows that the function~$h$ given by~\eqref{defn:functionh}
belongs to~$\overline{J}\setminus J$.

Set $Z = C(K)/J$, and let $Q\colon C(K)\to Z$ be the quotient
homomorphism. To prove that~$Z$ is not Archimedean, set $g_n =
(h-\frac{1}{n}\boldsymbol{1})^+\in C(K)^+$ for each $n\in\N$, where~$h$ is
defined by~\eqref{defn:functionh}, as above.  We see that $g_n\in J$
because~$h$ is continuous with $h(s_0)=0$. Hence the inequality
$\boldsymbol{1}\ge n(h-g_n)$ implies that $Q(\boldsymbol{1})\ge nQ(h)$ for
each $n\in\N$. However, $Q(h)\ne 0$ because $h\notin J$, and
therefore~$Z$ is not Archimedean.

Being a Banach lattice, $C(K)$ has a positively homogeneous continuous
function calculus which satisfies~\eqref{compositionsEq}. More
precisely, in the particular case of a $C(K)$-space, the positively
homogeneous continuous function calculus takes the following explicit
form:
\begin{equation}\label{phcfcforCKeq1} \bigl(\Phi_{\boldsymbol{f}}(g)\bigr)(t) 
= g\bigl(f_1(t),\ldots,f_n(t)\bigr)\qquad (t\in K) \end{equation} for
$n\in\N$, $\boldsymbol{f} = (f_1,\ldots,f_n)\in C(K)^n$ and $g\in
H_n$, as it is easy to see.

Let $R\colon Z\to C(K)$ be a right inverse map of~$Q$. (We can
choose~$R$ to be linear if we wish, but not in general a vector
lattice homomorphism, of course.)  Then, for each $n\in\N$ and $\boldsymbol{z} =
(z_1,\ldots,z_n)\in Z^n$, we can define a vector lattice homomorphism
by
\begin{equation}\label{phcfcforCKeq2}
 \Phi_{\boldsymbol{z}} = Q\circ \Phi_{(Rz_1,\ldots,Rz_n)}\colon H_n\to Z. \end{equation}
 It satisfies~\eqref{defnphcfcEq} because 
\[ \Phi_{\boldsymbol{z}}(\pi_i^{(n)}) = Q\bigl(\Phi_{(Rz_1,\ldots,Rz_n)}(\pi_i^{(n)})\bigr) 
= Q(Rz_i) = z_i\qquad (i\in\{1,\ldots,n\}). \]

To verify that~$\Phi_{\boldsymbol{z}}$
satisfies~\eqref{compositionsEq}, let $m\in\N$, $f_1,\ldots,f_m\in
H_n$ and $g\in H_m$, and set \mbox{$k_i =
  \Phi_{(Rz_1,\ldots,Rz_n)}(f_i)\in C(K)$} for
$i\in\{1,\ldots,m\}$. Then, on the one hand, \eqref{phcfcforCKeq2}
implies that
\begin{equation}\label{3MayEq1}
 \Phi_{(\Phi_{\boldsymbol{z}}(f_1),\ldots,\Phi_{\boldsymbol{z}}(f_m))}(g)
 = Q\bigl(\Phi_{(RQk_1,\ldots,RQk_m)}(g)\bigr), \end{equation} while on the
other,
\begin{equation}\label{3MayEq2}
\Phi_{\boldsymbol{z}}\bigl(g\circ(f_1\times\cdots\times f_m)\bigr) =
Q\bigl(\Phi_{(k_1,\ldots,k_m)}(g)\bigr) \end{equation}
by~\eqref{phcfcforCKeq2} and~\eqref{compositionsEq}.  Since $RQk_i -
k_i\in \ker Q = J$ for each $i\in\{1,\ldots,m\}$, we can find a
neigh\-bour\-hood~$U$ of~$s_0$ such that
\[ RQk_i(u) = k_i(u)\qquad (u\in U,\,i\in\{1,\ldots,m\}). \] 
Hence, using~\eqref{phcfcforCKeq1} twice, we obtain
\begin{align*} \bigl(\Phi_{(RQk_1,\ldots,RQk_m)}(g)\bigr)(u) &= 
g\bigl(RQk_1(u),\ldots,RQk_m(u)\bigr)\\ &=
g\bigl(k_1(u),\ldots,k_m(u)\bigr) =
\bigl(\Phi_{(k_1,\ldots,k_m)}(g)\bigr)(u)\qquad (u\in U),
 \end{align*}
so that $\Phi_{(RQk_1,\ldots,RQk_m)}(g) -
\Phi_{(k_1,\ldots,k_m)}(g)\in J$.  Combining this
with~\eqref{3MayEq1}--\eqref{3MayEq2}, we conclude
that~\eqref{compositionsEq} holds.
\end{proof}

\begin{remark}
\begin{enumerate}[label={\normalfont{(\roman*)}}]
\item The reader who prefers a more concrete example may simply consider $K = [0,1]$ and $s_0=0$ in Proposition~\ref{prop:twoexamples}. 
\item The hypothesis that $s_0$ is the accumulation point of a \textsl{countably infinite} subset of~$K$ is necessary, as the example $K = [0,\omega_1]$ (the set of ordinals no greater than the first uncountable ordinal~$\omega_1$, endowed with the order topology) and $s_0 = \omega_1$ shows: it is a standard fact that  the set~$X$ given by~\eqref{sublatticeX:Eq1} is equal to~$C(K)$ in this case.  
\end{enumerate}
\end{remark}

Our final example involves the following notion. 
For a locally compact Hausdorff space~$\Omega$,
\begin{multline*} C_0(\Omega) = \bigl\{ f\in C(\Omega)\midcolon
  \text{the set}\ \{ t\in\Omega\,:\,
  \abs{f(t)}\geqslant\varepsilon\}\ \text{is compact for
    each}\ \varepsilon>0\bigr\} \end{multline*} is a Banach lattice
with respect to the pointwise defined operations and the supremum
norm~$\norm{\,\cdot\,}_\infty$.  The evaluation map at a
point $t_0\in\Omega$ is denoted by~$\varepsilon_{t_0}$, that is,
\[ \varepsilon_{t_0}\colon\ f\mapsto f(t_0),\quad C_0(\Omega)\to\R. \] 
We require the following result, which is probably
well-known. However, as we have been unable to find it in the
literature, we include a proof, which is similar to that of
\cite[Theorem~2.33]{Aliprantis:06}.

\begin{lemma}\label{latticehomsfromC0K}
  Let~$\Omega$ be a locally compact Hausdorff space. 
  A map $\varphi\colon
C_0(\Omega)\to\R$ is a vector lattice homomorphism if and only if $\varphi
= c\,\varepsilon_{t_0}$ for some $c\in[0,\infty)$ and $t_0\in\Omega$.
  \end{lemma}
\begin{proof}  Each map of the
form~$c\,\varepsilon_{t_0}$, where $c\in[0,\infty)$ and
  $t_0\in\Omega$, is clearly a vector lattice homo\-mor\-phism.

  Conversely, suppose that $\varphi\colon C_0(\Omega)\to\R$ is a
  vector lattice homomorphism, and denote by~$C_c(\Omega)$ the
  norm-dense subspace of~$C_0(\Omega)$ consisting of all compactly
  supported functions.  The restriction of~$\varphi$ to~$C_c(\Omega)$
  is a positive linear functional, so the Riesz--Markov Theorem (see
  for instance \cite[p.~352]{Royden:88}) implies that
  \[ \varphi(f)=\int f\,\mathrm{d}\mu\qquad (f\in
  C_c(\Omega)) \] for some inner regular Borel measure~$\mu$
  on~$\Omega$. Let $\supp\mu$ be the support of~$\mu$, defined as in
  \cite[Exercise~24, pp.~351--352]{Royden:88}.  If
  $\supp\mu=\emptyset$, then $\varphi=0$, so we can take $c=0$ and any
  point~$t_0\in\Omega$. Otherwise the result will follow provided that
  we can show that $\supp\mu$ consists of a single point. Assume the
  contrary, so that~$\supp\mu$ contains two distinct points, say~$t_1$
  and~$t_2$. Since~$\Omega$ is locally compact and Haus\-dorff, there
  are disjoint open, relatively compact sub\-sets~$U_1$ and~$U_2$
  of~$\Omega$ such that $t_i\in U_i$ for $i\in\{1,2\}$. By Urysohn's
  Lemma, we can find continuous functions $f_1,f_2\colon\Omega\to
  [0,1]$ such that $f_i(t_i)=1$ and $f_i$ vanishes on~$\Omega\setminus
  U_i$ for $i\in\{1,2\}$. Note that $f_1,f_2\in C_c(\Omega)$
  because~$U_1$ and~$U_2$ are relatively compact. Hence, on the one
  hand, we have
  \[ \varphi(f_1)\wedge\varphi(f_2) = \int f_1\,\mathrm{d}\mu\wedge \int f_2\,\mathrm{d}\mu >0 \]
  because $t_1,t_2\in\supp\mu$. On the other, $f_1\wedge f_2=0$
  implies that $\varphi(f_1)\wedge\varphi(f_2)=\varphi(f_1\wedge
  f_2)=0$, which is clearly absurd.
\end{proof}
  
\begin{example}\label{ex:R2lexHasNoPhcfc}
  Endow the vector space $X=\R^2$ with the lexicographic order:
  \begin{equation} \label{ex:R2lexHasNoPhcfcEq1} (s_1,s_2)\leqslant
  (t_1,t_2)\ \Longleftrightarrow\ (s_1<t_1)\ \text{or}\ (s_1 =
  t_1\ \text{and}\ s_2\leqslant t_2)\qquad (s_1,s_2,t_1,t_2\in\R), \end{equation}
  and  set
  \begin{equation}\label{ex:R2lexHasNoPhcfceq1}
    x_1 = (1,0)\in X^+,\qquad x_2 = (0,1)\in X^+\qquad\text{and}\qquad
    \boldsymbol{x} = (x_1,x_2)\in X^2. \end{equation} Then $X$ is a
  vector lattice which is not Archimedean because $nx_2\leqslant x_1$
  for each $n\in\N$.

  We claim that~$X$ does not admit any positively homogeneous
  continuous function calculus. More precisely, we shall show that for
  the particular choice of~$\boldsymbol{x}$ given
  by~\eqref{ex:R2lexHasNoPhcfceq1}, no vector lattice homomorphism
  $\Phi_{\boldsymbol{x}}\colon H_2\to X$ with
  $\Phi_{\boldsymbol{x}}(\pi_i^{(2)})=x_i$ for $i\in\{1,2\}$
  exists.

 Assume the contrary, and set $\varphi_i =
  \pi_i^{(2)}\circ\Phi_{\boldsymbol{x}}\colon H_2\to\R$ for
  $i\in\{1,2\}$. The de\-fi\-ni\-tion~\eqref{ex:R2lexHasNoPhcfcEq1} implies
  that the first coordinate projection $\pi_1^{(2)}\colon X\to\R$ is a
  vector lattice homomorphism, and consequently~$\varphi_1$ is a
  vector lattice homomorphism.  Recalling the
  identification~\eqref{HnequalsCSlinftyn} of~$H_2$ with
  $C(S_{\ell_\infty^2})$ and using
  Lemma~\ref{latticehomsfromC0K}, we deduce
  that $\varphi_1 = c_1\,\varepsilon_s$ for some $c_1\in[0,\infty)$
    and $s\in S_{\ell_\infty^2}$.  Equation~\eqref{EqMaxofpiis} in the
    case $n=2$ implies that
  \begin{align*} c_1 &=  \varphi_1(\boldsymbol{1}) = 
\pi_1^{(2)}\circ\Phi_{\boldsymbol{x}}\bigl(\abs{\pi_1^{(2)}}\vee\abs{\pi_2^{(2)}}\bigr)\\ &=
\pi_1^{(2)}\bigl(\bigabs{\Phi_{\boldsymbol{x}}(\pi_1^{(2)})}\vee
\bigabs{\Phi_{\boldsymbol{x}}(\pi_2^{(2)})}\bigr) =
\pi_1^{(2)}\bigl(\abs{x_1}\vee\abs{x_2}\bigr) = \pi_1^{(2)}(x_1) =
1. \end{align*} Now the calculation
\[ \pi_i^{(2)}(s) =
  \varphi_1 (\pi_i^{(2)}) =
  \pi_1^{(2)}\bigl(\Phi_{\boldsymbol{x}}(\pi_i^{(2)})\bigr) =
  \pi_1^{(2)}(x_i) = \delta_{i,1}\qquad (i\in\{1,2\}) \] shows that $s =
  (1,0) = x_1$, and consequently $\varphi_1 = \varepsilon_{x_1}$.

Although~$\pi_2^{(2)}$ is not a vector lattice homomorphism, the
restriction of~$\varphi_2$ to~$\ker\varphi_1$ is a vector lattice
homomorphism because
\begin{align*}
    \varphi_2(f\vee g) &=
    \pi_2^{(2)}\bigl(\Phi_{\boldsymbol{x}}(f)\vee\Phi_{\boldsymbol{x}}(g)\bigr)
    = \pi_2^{(2)}\bigl((0,\varphi_2(f))\vee(0,\varphi_2(g))\bigr) =
    \varphi_2(f)\vee\varphi_2(g) \end{align*} for every
$f,g\in\ker\varphi_1$.  Since we can identify \mbox{$\ker\varphi_1 =
  \{ f\in C(S_{\ell_\infty^2})\midcolon f(x_1)=0\}$} with
\mbox{$C_0\bigl(S_{\ell_\infty^2}\setminus\{x_1\}\bigr)$},
Lemma~\ref{latticehomsfromC0K} implies that
$\varphi_2\!\!\upharpoonright_{\ker\varphi_1} = c_2\,\varepsilon_t$
for some $c_2\in[0,\infty)$ and $t=(t_1,\,t_2)\in
  S_{\ell_\infty^2}\setminus\{x_1\}$.  We have
  $\pi_2^{(2)}\in\ker\varphi_1$ because $\pi_2^{(2)}(x_1) = 0$, and
  therefore
 \[  c_2t_2 = \varphi_2(\pi_2^{(2)}) = \pi_2^{(2)}\bigl(\Phi_{\boldsymbol{x}}(\pi_2^{(2)})\bigr) = \pi_2^{(2)}(x_2) = 1, \]   
so that $c_2,t_2>0$. Take $\lambda\in (t_1/t_2,\infty)$ and
define $f = \pi_1^{(2)}\vee (\lambda\pi_2^{(2)}) - \pi_1^{(2)}\in C(S_{\ell_\infty^2})$. Then on the one hand we have 
 \[
\varphi_2(f) = \pi_2^{(2)}\bigl(\Phi_{\boldsymbol{x}}(f)\bigr) = 
\pi_2^{(2)}\bigl(x_1\vee(\lambda x_2)-x_1\bigr) =
0,
\]
while on the other
$f\in\ker\varphi_1$ because  
$f(x_1) = 0$, 
and hence
 \[
\varphi_2(f) = c_2f(t) = c_2(t_1\vee(\lambda t_2)-t_1) = c_2(\lambda t_2-t_1)>0.  \] 
This is clearly absurd, and this
contradiction completes the proof of our claim that~$X$ does not admit a positively
homogeneous continuous function calculus.
\end{example}  

\subsection*{Acknowledgements} We are grateful to Professor Gerard Buskes (Mississippi) for discussing a preliminary version of this work with us.

\end{document}